\documentclass[reqno, 11pt, oneside]{amsart}
\usepackage{latexsym, amsmath,amssymb}
\usepackage{graphicx}
\usepackage{xcolor}
\usepackage{enumitem}
\usepackage[utf8]{inputenc}

 \numberwithin{equation}{section}

\usepackage{amsmath}

\def\XXint#1#2#3{{\setbox0=\hbox{$#1{#2#3}{%
\int}$ }
\vcenter{\hbox{$#2#3$ }}\kern-.6\wd0}}

\renewcommand{\epsilon}{\varepsilon}
\newtheorem{theorem}{Theorem}

\newtheorem{lemma}[theorem]{Lemma}
\newtheorem{corollary}[theorem]{Corollary}

\newtheorem{proposition}[theorem]{Proposition}

\newtheorem{remark}[theorem]{Remark}
\newcommand{\bth}{\begin{theorem}}
\newcommand{\ble}{\begin{lemma}}
\newcommand{\bcor}{\begin{corr}}

\newcommand{\bdeff}{\begin{deff}}

\newcommand{\bprop}{\begin{proposition}}
\newcommand{\ele}{\end{lemma}}
\newcommand{\ecor}{\end{corr}}
\newcommand{\edeff}{\end{deff}}

\numberwithin{theorem}{section}

\newcommand{\eprop}{\end{proposition}}

\renewcommand{\l}{\lambda}

\renewcommand{\Pi}{\varPi}

\renewcommand{\epsilon}{\varepsilon}

\newcommand{\R}{{\mathbb R}}

\newcommand{\norm}[1]{\left\|#1\right\|}

\usepackage{esint}

\usepackage{comment}

\usepackage[
colorlinks=true,
linkcolor=blue,
 citecolor=blue,
  urlcolor=blue,
     pagebackref,
]{hyperref}

\allowdisplaybreaks[3]

\def\vint_#1{\mathchoice%
        {\mathop{\kern 0.2em\vrule width 0.6em height 0.69678ex depth -0.58065ex
                \kern -0.8em \intop}\nolimits_{\kern -0.4em#1}}%
        {\mathop{\kern 0.1em\vrule width 0.5em height 0.69678ex depth -0.60387ex
                \kern -0.6em \intop}\nolimits_{#1}}%
        {\mathop{\kern 0.1em\vrule width 0.5em height 0.69678ex
            depth -0.60387ex
                \kern -0.6em \intop}\nolimits_{#1}}%
        {\mathop{\kern 0.1em\vrule width 0.5em height 0.69678ex depth -0.60387ex
                \kern -0.6em \intop}\nolimits_{#1}}}
\def\vintslides_#1{\mathchoice%
        {\mathop{\kern 0.1em\vrule width 0.5em height 0.697ex depth -0.581ex
                \kern -0.6em \intop}\nolimits_{\kern -0.4em#1}}%
        {\mathop{\kern 0.1em\vrule width 0.3em height 0.697ex depth -0.604ex
                \kern -0.4em \intop}\nolimits_{#1}}%
        {\mathop{\kern 0.1em\vrule width 0.3em height 0.697ex depth -0.604ex
                \kern -0.4em \intop}\nolimits_{#1}}%
        {\mathop{\kern 0.1em\vrule width 0.3em height 0.697ex depth -0.604ex
                \kern -0.4em \intop}\nolimits_{#1}}}

\newcommand{\aveint}[2]{\mathchoice%
        {\mathop{\kern 0.2em\vrule width 0.6em height 0.69678ex depth -0.58065ex
                \kern -0.8em \intop}\nolimits_{\kern -0.45em#1}^{#2}}%
        {\mathop{\kern 0.1em\vrule width 0.5em height 0.69678ex depth -0.60387ex
                \kern -0.6em \intop}\nolimits_{#1}^{#2}}%
        {\mathop{\kern 0.1em\vrule width 0.5em height 0.69678ex depth -0.60387ex
                \kern -0.6em \intop}\nolimits_{#1}^{#2}}%
        {\mathop{\kern 0.1em\vrule width 0.5em height 0.69678ex depth -0.60387ex
                \kern -0.6em \intop}\nolimits_{#1}^{#2}}}

\def\XXint#1#2#3{{\setbox0=\hbox{$#1{#2#3}{\int}$}
    \vcenter{\hbox{$#2#3$}}\kern-.5\wd0}}

\newcommand{\vertiii}[1]{{\left\vert\kern-0.25ex\left\vert\kern-0.25ex\left\vert #1 
    \right\vert\kern-0.25ex\right\vert\kern-0.25ex\right\vert}}

\newcommand{\vertii}[1]{{\left\vert\kern-0.25ex\left\vert\kern-0.25ex  #1 
    \kern-0.25ex\right\vert\kern-0.25ex\right\vert}}

\makeatletter
\@namedef{subjclassname@2020}{\textup{2020} Mathematics Subject Classification}
\makeatother

\usepackage{mathabx}
\usepackage{geometry}

\begin{document}

\title[Asymptotically Sharp Embedding of $A_\infty$ into $A_p$]
{Asymptotically sharp embedding of $A_\infty$ into $A_p$ for flat weights and applications to Poincar\'e--Sobolev inequalities}

\author[A. Claros]{Alejandro Claros}

\address[A. Claros]{BCAM -- Basque Center for Applied Mathematics, Bilbao, Spain\newline
Universidad del Pa\'is Vasco / Euskal Herriko Unibertsitatea (UPV/EHU), Bilbao, Spain}

\email{aclaros@bcamath.org, aclaros003@ikasle.ehu.eus}

\author[E. Rela]{Ezequiel Rela}
\address[E. Rela]{Universidad Torcuato Di Tella. Departamento de Matemáticas y Estadística and CONICET}
\email{ezequiel.rela@utdt.edu}

\thanks{A. Claros is supported by the Basque Government through the BERC 2022-2025 program, by the Ministry of Science and Innovation through Grant PRE2021-099091 funded by BCAM Severo Ochoa accreditation CEX2021-001142-S/MICIN/AEI/10.13039/501100011033 and by ESF+, and by the project PID2023-146646NB-I00 funded by MICIU/AEI/10.13039/501100011033 and by ESF+. 
}





\begin{abstract}
We provide new quantitative results on the embedding of the Muckenhoupt class $A_\infty$ into $A_p$ with the correct asymptotic behavior when the Fujii--Wilson constant $[w]_{A_\infty}$ is close to 1, namely that the parameter $p$ goes to 1 when the weight is nearly constant. As intermediate steps towards the result, we obtain quantitative estimates on the weighted and unweighted BMO norms of $\log w$ for an $A_\infty$ weight $w$. As a consequence, we show that a precise quantitative weighted Poincar\'e--Sobolev inequality can be proved for weights with small $[w]_{A_\infty}$ that recovers the classical Sobolev exponent $p^*=\frac{np}{n-p}$ when $[w]_{A_\infty}\to 1^+$. 
\end{abstract}

\maketitle

\section{Introduction and Main Results}

The purpose of this article is to provide new, improved quantitative estimates for the embedding of $A_\infty$ weights into $A_p$ classes.  Throughout, by a weight we mean a locally integrable function $w$ such that $0<w(x)<\infty$ for almost every $x\in\mathbb{R}^n$.  For $1< p<\infty$, the Muckenhoupt $A_p$ classes of weights are defined by the condition 
$$[w]_{A_p}:=\sup _Q\left(\frac{1}{|Q|} \int_Q w(y) \, d y\right)\left(\frac{1}{|Q|} \int_Q w(y)^{1-p^{\prime}} \, d y\right)^{p-1}<\infty,
$$
where the supremum is taken over all the cubes  $Q$  in $\mathbb{R}^n$ with sides parallel to the coordinate axes. The relevance of this class of weights comes from the very well-known fact that they characterize the weighted $L^p$ boundedness of maximal and singular operators like the Hardy--Littlewood maximal operator and Calder\'on--Zygmund operators. It is worth mentioning that when deriving the $A_p$ condition above as a necessary condition for the boundedness of the H-L maximal function $M$, it is immediate to verify that any meaningful $A_p$ weight $w$ should be \emph{positive} almost everywhere and locally integrable (see, for example, classical references such as \cite{Duo01,GCRdF,GrafakosMF}).
The limiting case $p=1$ defines the class $A_1$; that is, the set of weights $w$ such that
$$
[w]_{A_1}:=\sup_Q\left(\frac{1}{|Q|} \int_Q  w(y)\, dy\right) \underset{Q}{\operatorname{ess} \sup }\left(w^{-1}\right)<+\infty.
$$
At the other endpoint we have the $A_\infty$ class defined as the union of $A_p$ classes, namely $A_\infty:=\bigcup_{p\ge1} A_p.$
The membership in $A_\infty$ can also be characterized in terms of the finiteness of the so-called $A_\infty$ constant defined as 
$$
[w]'_{A_{\infty}}:=\sup _Q\left(\frac{1}{|Q|} \int_Q w\right) \exp \left(\frac{1}{|Q|} \int_Q \log w^{-1}\right).
$$
This $A_\infty$ constant was originally introduced by Hru\v{s}\v{c}ev \cite{Hruscev} (see also \cite{GCRdF}). In recent years, a new definition was introduced that is more adequate to express quantitative norm inequalities for maximal and singular operators, namely the one defined as
\begin{equation*}
  [w]_{A_\infty}:=\sup_Q\frac{1}{w(Q)}\int_Q M(w\chi_Q)(x)\, dx,
\end{equation*}
where $M$ stands for the Hardy--Littlewood maximal function,
$$
M f(x)=\sup _{Q \ni x} \frac{1}{|Q|} \int_Q |f(y)|\,  dy.
$$
Note that the weights under consideration are positive a.e., so the quotient in the definition of $[w]_{A_\infty}$ is well defined. Moreover, since $M(w\chi_Q)(x)\ge w_Q$ for every $x\in Q$, we have $[w]_{A_\infty}\ge 1$.  This is sometimes called the ``Wilson constant'', as it was implicitly introduced by Wilson a long time ago (see \cite{Wilson:87, Wilson:89, Wilson-LNM}) with a different notation,  and was also implicitly introduced by Fujii \cite{Fujii}. Its explicit formulation appears in \cite{HP}. This constant  is more relevant since it is known that $ [w]'_{A_\infty}<\infty \iff [w]_{A_\infty}<\infty$, but there are examples of weights $w \in A_{\infty}$ so that   $[w]_{A_\infty} $ is much smaller than
$[w]'_{A_\infty}$ (interesting examples can be found in \cite{HP},  where it was shown that $[w]'_{A_\infty}$ can be exponentially larger than $[w]_{A_\infty}$ ).

An easy result is that the $A_p$ classes are increasing in $p$ and \emph{open}, in the sense that if $w\in A_p$ with $p>1$ then there exists an $\varepsilon>0$ such that $w\in A_{p-\varepsilon}$, with an explicit control of the $A_{p-\varepsilon}$ constant $[w]_{A_{p-\varepsilon}}$  in terms of $[w]_{A_p}$. The relevant question here is to determine, given the weight $w\in A_\infty$ with a certain constant $[w]_{A_\infty}$, what is the best possible $p\ge 1$ such that $w\in A_p$. Moreover, it is interesting to find a quantitative connection between this optimal $p\ge 1$ and the constant $[w]_{A_\infty}$.

Work in this direction can be found in \cite{HagPar-Embeddings}, where, as a consequence of their study of quantitative weighted Solyanik estimates,  the authors prove that if $w$ is an $A_{\infty}$ weight in $\mathbb{R}^{n}$, then $w \in A_p$ for $p> e^{c_n[w]_{A_{\infty}}}$ with $[w]_{A_p} \leq e^{e^{c_n[w]_{A_{\infty}}}}$, where $c_n>1$ is a dimensional constant. This seems to be the best available result in the literature on this question, but a simple observation is that both these estimates fall short of what they should be in the particular case of weights with $[w]_{A_\infty}\to 1$. Notice that since $[w]_{A_\infty}=1$ is equivalent to $w$ being constant almost everywhere (see for instance \cite[Proposition 2.1]{ParissisRela}), it makes sense to refer to this asymptotic analysis as the case of \emph{flat weights}. The main contribution in this article is to provide a sharp asymptotic estimate for the $A_\infty \hookrightarrow A_p$ embedding for flat weights. That is, denoting by $p([w]_{A_\infty})$ the index such that $w\in A_{p([w]_{A_\infty})}$, we show that $p([w]_{A_\infty})\to 1$ when $[w]_{A_\infty}\to 1$, with explicit quantitative estimates for the $A_p$ constant. The remarkable aspect is that we indeed need to restrict ourselves to weights with small $A_\infty$ constants, a condition that seems to be common in results of this type. We do not know whether similar estimates can be proved for \emph{all} $A_\infty$ weights. 

 We now state the main result of the paper.

\begin{theorem}\label{thm:Ainfty Ap}
	There exist dimensional constants $c_n, \tau_n, \tau_n'>0$ such that if $w\in A_\infty$ with $[w]_{A_\infty} \le 1+ c_n$ then 
	\begin{equation*}
		w\in A_p \qquad \text{with} \qquad p=1+\tau _n ([w]_{A_\infty} -1), 
	\end{equation*}
	and moreover,
	\begin{equation*}
		[w]_{A_p}\le \exp (\tau_n'  ([w]_{A_\infty} -1) ) .
	\end{equation*}
\end{theorem}

\begin{remark}
	We observe that under the smallness condition $[w]_{A_\infty} \le 1+c_n$ of the preceding theorem, the exponential upper bound can be controlled linearly. As a result, in the flat regime, the quantities $p$ and $[w]_{A_p}$ are comparable, yielding the relation
	\begin{equation*}
		p \simeq_n [w]_{A_p} \simeq_n 1+\tau_n ([w]_{A_\infty} -1).
	\end{equation*}
\end{remark}

A key intermediate step towards the proof of Theorem \ref{thm:Ainfty Ap} is an estimate on the BMO seminorm of  $\log w$, where  $w$ is an $A_\infty$ weight. The fact that $\log w$ belongs to BMO when $w\in A_\infty$ is not new; even some quantitative estimates are known. A nice result in this area related to the $[w]'_{A_\infty}$ constant was provided by Korey \cite{Korey}. Among other results, Korey proved that 
$$
\|\log w\|_{\operatorname{BMO}}\le \log\left(2[w]'_{A_\infty}\right).
$$
There is also an asymptotic estimate of the form
$$
\|\log w\|_{\operatorname{BMO}}\le C\sqrt{\log\left([w]'_{A_\infty}\right)}\qquad \text{as} \qquad [w]'_{A_\infty}\to 1.
$$
Results like these were later used by Pattakos and Volberg in \cite{pattakos-MRL} to study the interesting problem of the continuity of norm estimates in weighted inequalities for Calder\'on--Zygmund operators.

In that direction, we provide here a sharper estimate for the BMO seminorm in terms of a mixed combination of the Hru\v{s}\v{c}ev and Fujii--Wilson $A_\infty$ constants. This estimate should be compared with the result in \cite[Proposition 1.21]{HP}. 

\begin{theorem}\label{thm:log w in BMO}
	If $w \in A_\infty$, then $\log w \in \operatorname{BMO}$ with
	\begin{equation*}
		\norm{\log w}_{\operatorname{BMO}} \le c_n \log(2[w]_{A_\infty}') ([w]_{A_\infty} - 1),
	\end{equation*}
	for some dimensional constant $c_n > 0$.
\end{theorem}

As a final but important contribution, we mention that the sharp embedding theorem for flat weights produces an interesting improvement of the recent Poincar\'e--Sobolev-type inequalities obtained by the first author in \cite{Claros25}. 

\begin{theorem}\label{thm:PoincareSobolev-SharpFlat2}
	Let $n\ge 2$ and let $1< p <n$. Let $w\in A_\infty$ and consider the exponent $p^*_w$ defined by the relation
	\begin{equation*}
		\frac{1}{p}-\frac{1}{p^*_w}= \frac{1}{n}\frac{1}{1+\tau_n ([w]_{A_\infty}-1)}.
	\end{equation*}
	There exists $c_{n,p}>0$ such that if $[w]_{A_\infty} \le 1+ c_{n,p}$, then we have
	\begin{equation*}
		\left( \frac{1}{w(Q)}\int_Q |f-f_{Q}|^{p^*_w}w\right)^\frac{1}{p^*_w}\le c_n p^* (1+\tau_n ([w]_{A_\infty}-1))^\frac{1}{p} \ell(Q) \left( \frac{1}{w(Q)}\int_Q |\nabla f|^p w\right)^\frac{1}{p},
	\end{equation*}
	for every cube $Q$. Moreover, we have
	\begin{equation*}
		 \left\| f-f_{Q}\right\|_{L^{p^*_w, p}\left(Q, \frac{w(x)dx}{w(Q)}\right)}\le c_n p^* (1+\tau_n ([w]_{A_\infty}-1))^\frac{1}{p} \ell(Q) \left( \frac{1}{w(Q)}\int_Q |\nabla f|^p w\right)^\frac{1}{p},
	\end{equation*}
	for every cube $Q$.
\end{theorem}

Although the previous result is restricted to weights with a small $A_\infty$ constant, it has the remarkable consequence of asymptotically recovering the classical
Sobolev exponent when $[w]_{A_\infty}\to 1$. That is, when $[w]_{A_\infty}\to 1$, we have $p^*_w \to p^*=\frac{np}{n-p}$.

To put this result into perspective, it is worth noting that the weighted Poincar\'e inequality does not hold for the entire class $A_\infty$, as shown in \cite[Theorem 1.26]{PR-Poincare}.
Consequently, the Poincar\'e--Sobolev inequality does not hold either. In contrast, our result shows that if the $A_\infty$ constant of $w$ is not too large, that is, if $w$ is not a highly degenerate weight, then a weighted Poincar\'e--Sobolev inequality can be established. In other words, the failure of a Poincar\'e inequality for the entire $A_\infty$ class comes precisely from weights with large $A_\infty$ constants, namely, the most singular weights in the class $A_\infty$.

\subsection*{Outline of the paper}
The article is organized as follows. In Section \ref{sec:Ainfty-Doubling} we present a new result on doubling properties of $A_\infty$ weights, which has the advantage of being asymptotically correct for flat weights. In Section \ref{sec:logw BMO} we present the proof of the intermediate result in Theorem \ref{thm:log w in BMO}. Finally, Section \ref{sec:Ainfty Ap} contains the proof of our main contribution in Theorem \ref{thm:Ainfty Ap}, while Section \ref{sec:PS} details its consequences for Poincar\'e--Sobolev inequalities as presented in Theorem \ref{thm:PoincareSobolev-SharpFlat2}.

\section{From \texorpdfstring{$A_\infty$}{Ainfty} to doubling}\label{sec:Ainfty-Doubling}

We include in this section a result that can be seen as a byproduct of the approach to studying properties of flat weights. We show that if a weight $w$ belongs to $A_\infty$, then the measure $w\, dx$ is doubling. Moreover, we obtain an explicit estimate of the doubling constant in terms of the Fujii--Wilson $A_\infty$ constant of $w$. The novelty here is that the doubling constant shows the correct asymptotic behavior for flat weights. 

 We will use the following reverse H\"older inequality from \cite{ParissisRela}, which is asymptotically sharp for flat weights. Let us mention that, in the case of $A_1$ weights, the first sharp reverse H\"older inequality was proved by Kinnunen \cite{Kinnunen-Dissertationes} in $\mathbb R^n$ with Lebesgue measure, and later extended to general measures in \cite{LuPeRe-strong}. 

\begin{theorem}\label{thm:RHIflat}
	Let $w\in A_\infty$. Then for each $0\le \varepsilon\le \frac{1}{2^{n+1}([w]_{A_\infty} -1)}$,
	\begin{equation*}
		\frac{1}{|Q|}\int_Q w^{1+\varepsilon} \le 2 [w]_{A_\infty}\left(  \frac{1}{|Q|}\int_Q w\right)^{1+\varepsilon},
	\end{equation*}
	for every cube $Q$. 
\end{theorem}

\begin{remark}
	The previous result is not exactly the same as the one stated in \cite{ParissisRela}. Here we slightly weaken the range of admissible $\varepsilon$, but in exchange we obtain a much more manageable constant, namely $2 [w]_{A_\infty}$, on the right-hand side.
\end{remark}

As a consequence, using the standard techniques, we obtain the following corollary.

\begin{corollary}\label{cor:Ainfty->subsets}
	Let $w\in A_\infty$. Then for each cube $Q$ and each measurable subset $E\subset Q$, we have
	\begin{equation*}
		\frac{w(E)}{w(Q)}\le 2[w]_{A_\infty} \left( \frac{|E|}{|Q|}\right)^{\frac{1}{1+2^{n+1}([w]_{A_\infty}-1)}}.
	\end{equation*}
\end{corollary}

From Theorem~\ref{thm:RHIflat} we also obtain a left-openness property of the $A_p$ classes. More precisely, given $w \in A_p$, we deduce that $w \in A_{p-\varepsilon}$ for some $\varepsilon>0$ depending on $[\sigma]_{A_\infty}$,  where $\sigma:=w^{1-p'}$ is the dual weight of $w$. Compared with the result in \cite{HPR1}, the value $p-\varepsilon$ that we obtain here is slightly smaller, but on the other hand, the constant is controlled by a mixed-type constant.  Furthermore, this estimate captures the correct asymptotic behavior for flat weights, as we have $p-\varepsilon \to 1$ when $[\sigma]_{A_\infty} \to 1$.

\begin{corollary}\label{cor:Ap-leftopen}
	Let $w\in A_p$ for some $1<p<\infty$. If we consider 
	\begin{equation}\label{eq:open-eps}
		\varepsilon = \frac{p-1}{2^{n+1}([\sigma]_{A_\infty} -1)+1},
	\end{equation}
	then $w\in A_{p-\varepsilon}$ with 
	\begin{equation*}
		[w]_{A_{p-\varepsilon}} \le (2[\sigma]_{A_\infty})^{p-1} [w]_{A_p}. 
	\end{equation*}
\end{corollary}

We next show that if a weight belongs to $A_\infty$, then the associated measure $w\, dx$ is doubling, and we provide an explicit estimate of the doubling constant in terms of the Fujii--Wilson $A_\infty$ constant of $w$.  A related quantitative doubling estimate in terms of the Fujii--Wilson $A_\infty$ constant was obtained in \cite[Corollary 3.6(i)]{HagPar-Embeddings}; the novelty of the estimate below is its correct asymptotic behavior for flat weights. 
Related estimates have also been obtained using Hru\v{s}\v{c}ev's exponential $A_\infty$ constant (see, for instance, \cite{GrafakosCF}).

\begin{theorem}\label{thm:Ainfty-doubling}
	Let $w\in A_\infty$. Then the measure $w\, dx$ is doubling, i.e., 
	\begin{equation*}
		w(2Q)\le D_w w(Q),
	\end{equation*}
	for every cube $Q$, where
	\begin{equation*}
		D_w \le \exp \left(\kappa \,  n  (4[w]_{A_\infty})^{1+2^{n+1}([w]_{A_\infty}-1)}\right),
	\end{equation*}
	and $\kappa$ is a dimension-free constant. 
\end{theorem}

\begin{remark}\label{rem: doubling}
	Using the same argument together with the reverse H\"older inequality from \cite{HPR1} produces the bound  $D_w \lesssim \exp (\exp (c_n [w]_{A_\infty}))$. 
\end{remark}

\begin{proof}
	Let $Q$ be a cube and let $\lambda>1$. Applying Corollary \ref{cor:Ainfty->subsets} to the cube $\lambda Q$ and the subset $E=\lambda Q \setminus Q$, we obtain
	\begin{align*}
		\frac{w( \lambda Q \setminus Q) }{w(\lambda Q )} \le & 2[w]_{A_\infty} \left( \frac{|\lambda Q \setminus Q|}{|\lambda Q|}\right)^\theta  \\
		= & 2[w]_{A_\infty} \left( 1- \frac{| Q|}{|\lambda Q|}\right)^\theta  \\
		= & 2[w]_{A_\infty} \left( 1- \lambda^{-n} \right)^\theta ,
	\end{align*}
	where $\theta= \frac{1}{1+2^{n+1}([w]_{A_\infty}-1)} \in (0,1)$. Hence,
	\begin{align*}
		w(\lambda Q) = & w(Q) + w(\lambda Q \setminus Q)\\
		\le & w(Q) + 2[w]_{A_\infty} \left( 1- \lambda^{-n} \right)^\theta w(\lambda Q).
	\end{align*}
	We now choose $\lambda >1$ such that 
	\begin{equation*}
		2[w]_{A_\infty} \left( 1- \lambda^{-n} \right)^\theta \le \frac{1}{2},
	\end{equation*}
	and consequently,
	\begin{equation*}
		w(\lambda Q)\le 2w(Q).
	\end{equation*}
	The previous inequality holds for any $1<\lambda \le \lambda_0$, where 
	\begin{equation*}
		\lambda_0 = \frac{1}{\left( 1- \left( 4[w]_{A_\infty}\right) ^{-\frac{1}{\theta} } \right)^{\frac{1}{n}}}.
	\end{equation*}
	Let $m\in \mathbb{N}$ be the smallest integer such that $\lambda_0 ^m \ge 2$. Then, we have
	\begin{equation*}
		w(2Q)\le w(\lambda_0^m Q) \le 2^m w(Q). 
	\end{equation*}
	We observe that $m\simeq \frac{\log 2}{\log \lambda_0}$, so it remains to estimate $\log \lambda_0$. Since
	\begin{equation*}
		\log \lambda_0= -\frac{1}{n} \log (1-s),
	\end{equation*}
	 where $s=(4[w]_{A_\infty})^{-\frac{1}{\theta}} \in (0,\frac{1}{4}]$, we can use the elementary inequality
	\begin{equation*}
		s\le -\log (1-s) \le \frac{s}{1-s} \le 2s ,
	\end{equation*}
	which gives  $\log \lambda_0 \simeq \frac{s}{n}$. Therefore, 
	\begin{align*}
		D_w \le & 2^m  \lesssim 2^{\frac{\log 2}{\log \lambda_0}} \le \exp \left( \frac{\log 2}{\log_2 e} \frac{n}{s} \right) \le \exp \left( \kappa \, n \left( 4[w]_{A_\infty} \right)^{1+2^{n+1}([w]_{A_\infty}-1)} \right).
	\end{align*}
\end{proof}

\begin{remark}
In the case of constant weights ($[w]_{A_\infty} = 1$), the doubling constant obtained above depends exponentially on the dimension, which is expected.  Moreover, for flat weights of the form $[w]_{A_\infty}=1+\delta$ with $\delta\to 0$, the upper bound for $D_w$ obtained above converges to $e^{4\kappa n}$. 
\end{remark}

\section{Logarithm of an \texorpdfstring{$A_\infty$}{Ainfty} weight}\label{sec:logw BMO}

In this section, we prove that if $w \in A_\infty$, then $\log w\in \operatorname{BMO}$. Moreover, we provide an explicit bound for the $\operatorname{BMO}$ seminorm of $\log w$ in terms of the $A_\infty$ constant of $w$. For this purpose, we introduce the following logarithmic $A_\infty$ constant to be used throughout the section:
\begin{equation*}
	[w]_{A_\infty}^{\log } :=\sup_Q \frac{1}{w(Q)}\int_Q \left(1+\log^+ \left( \frac{w(x)}{w_Q}\right) \right)w(x)\, dx,
\end{equation*}
where the supremum is taken over all cubes $Q$, and $w_Q = \frac{1}{|Q|}\int_Q w(x)dx$.

\begin{remark}
It is shown in \cite{HP} that $[w]_{A_\infty}^{\log }\le 2^{n+1} [w]_{A_\infty}$. In addition, by \cite{Stein-LlogL} we have the reverse inequality,
	\begin{equation*}
		\int_Q M(w\chi_Q)(x)dx\le 3^n \int_Q \left(1+\log ^+\left( \frac{w(x)}{w_Q}\right) \right) w(x) \, dx, 
	\end{equation*}
	for each cube $Q$. In conclusion, we have $[w]_{A_\infty}^{\log }\simeq_n  [w]_{A_\infty}$.
\end{remark}

As a first step, we show that $\log w$ belongs to the weighted $\operatorname{BMO}$ space, $\mathrm{BMO}_w$. This space, which is  usually   denoted as $\mathrm{BMO}_{w,w}$ in the literature (see, for instance, \cite{OPRRR}), is equipped with the seminorm
\begin{equation*}
    \|f\|_{\operatorname{BMO}_w} := \sup_Q \frac{1}{w(Q)} \int_Q |f(x) - f_{Q,w}| w(x)\,dx,
\end{equation*}
where the supremum is taken over all cubes $Q$, and $f_{Q,w} = \frac{1}{w(Q)} \int_Q f(x) w(x)\,dx$ is the weighted average of $f$ on the cube $Q$.

\begin{proposition}
	If $w\in A_\infty$, then $\log w \in \operatorname{BMO}_w$ and 
	\begin{equation}\label{eq:BMOw - Ainfty-log}
		\norm{\log w }_{\operatorname{BMO}_w} \le 8 ( [w]_{A_\infty}^{\log} -1 ). 
	\end{equation}
\end{proposition}

\begin{proof}
	Let $w\in A_\infty$ and let $Q$ be a cube.  Since $[w]_{A_\infty}^{\log}<\infty$ and $t\log(1/t)\le e^{-1}$ for $0< t \le 1$, we have $\log w\in L^1_{\mathrm{loc}}(w\,dx)$. Consequently, $\log w\in L^1(Q,w\,dx)$ and the weighted average $(\log w)_{Q,w}$ is well defined.   Since $\log w - (\log w)_{Q, w}$ has zero integral on the cube $Q$ with respect to the weighted measure $w\, dx$, we have
	\begin{align*}
		\frac{1}{w(Q)} \int_Q |\log w - (\log w)_{Q, w}| w = & \frac{2}{w(Q)} \int_Q \left(\log w - (\log w)_{Q, w}\right)^+ w \\
		\le & \frac{2}{w(Q)} \int_Q \left(\log w - \log w_Q\right)^+ w \\
		& + 2 \left(\log w_Q - (\log w)_{Q, w} \right)^+\\
		= & I_1 + I_2,
	\end{align*}
	where $a^+= \max (a,0)$. We can easily estimate the first term
	\begin{align*}
		I_1 = & \frac{2}{w(Q)} \int_Q \log ^+\left( \frac{w(x)}{w_Q}\right) w(x)\,dx \\
		= & 2\left( \frac{1}{w(Q)} \int_Q \left( 1 + \log ^+\left( \frac{w(x)}{w_Q}\right)\right) w(x)\,dx - 1 \right) \\
		\le & 2 ( [w]_{A_\infty}^{\log} -1) .
	\end{align*}
	
	On the other hand,
	\begin{align*}
		(\log w)_{Q, w} - \log w_Q = & \frac{1}{w(Q)}\int_Q (\log w(x)- \log w_Q )w(x)\,dx\\
		= & \frac{1}{w(Q)}\int_Q (\log w(x)- \log w_Q )^+ w(x)\,dx\\
		& - \frac{1}{w(Q)}\int_Q (\log w_Q - \log w(x) )^+ w(x)\,dx\\
		= & P-N.
	\end{align*}
	
	We estimate $P$ and $N$ separately. The bound for $P$ is analogous to that for $I_1$,
	\begin{align*}
		P = & \frac{1}{w(Q)} \int_Q \log ^+\left( \frac{w(x)}{w_Q}\right) w(x)\,dx \\
		= & \frac{1}{w(Q)} \int_Q \left( 1 + \log ^+\left( \frac{w(x)}{w_Q}\right)\right) w(x)\,dx - 1 \\
		\le & [w]_{A_\infty}^{\log} -1 .
	\end{align*}
	To estimate $N$, we use the elementary inequality $\log t \le t-1$,
	\begin{align*}
		N = & \frac{1}{w(Q)}\int_Q \left( \log \left( \frac{w_Q}{w(x)}\right) \right)^+ w(x)\,dx \\
		\le & \frac{1}{w(Q)}\int_{ \{ w\le w_Q\} } \left( \frac{w_Q}{w(x)}-1\right) w(x)\,dx \\
		= & \frac{1}{w(Q)}\int_{ Q } \left( w_Q - w(x) \right)^+ \,dx \\
		\le & \frac{1}{w(Q)}\int_{ Q } | w(x) - w_Q | \,dx \\
		= & \frac{2}{w(Q)}\int_{ Q } ( w(x) - w_Q )^+ \,dx \\
		= & \frac{2}{w(Q)}\int_{ \{ w\ge w_Q\} } w_Q \left( \frac{w(x)}{w_Q} - 1 \right) \,dx \\
		\le & \frac{2}{w(Q)}\int_{ \{ w\ge w_Q\} } w_Q \frac{w(x)}{w_Q} \log \left( \frac{w(x)}{w_Q} \right) \,dx \\
		= & \frac{2}{w(Q)}\int_{ Q } \log^+ \left( \frac{w(x)}{w_Q} \right) w(x)\,dx \\
		\le & 2([w]_{A_\infty}^{\log} -1),
	\end{align*}
	where in the third inequality we use $t-1\le t\log t$ for $t\ge 1$. Therefore,
	\begin{align*}
		I_2 = & 2 \left(\log w_Q - (\log w)_{Q, w} \right)^+ \\
		\le & 2 |\log w_Q - (\log w)_{Q, w} |\\
		\le & 2 |(\log w)_{Q, w} -\log w_Q |\\
		= & 2 |P-N| \\
		\le & 2P+2N \\
		\le & 2([w]_{A_\infty}^{\log} -1) + 4([w]_{A_\infty}^{\log} -1)\\
		= & 6 ([w]_{A_\infty}^{\log} -1).
	\end{align*}
	Combining the previous estimates, we have
    \begin{align*}
		\frac{1}{w(Q)} \int_Q |\log w - (\log w)_{Q, w}| w \le 8 ([w]_{A_\infty}^{\log} -1),
	\end{align*}
    for every cube $Q$. Taking the supremum over all cubes $Q$ completes the proof.
\end{proof}

\begin{remark}
	A closer inspection of the argument shows that we have, in fact, proved a stronger statement. For each cube $Q$, we have
	\begin{equation*}
		\frac{1}{w(Q)} \int_Q |\log w(x) - (\log w)_{Q, w}| w(x)\,dx  \le  \frac{8}{w(Q)} \int_Q \log ^+\left( \frac{w(x)}{w_Q}\right) w(x)\,dx.
	\end{equation*}
\end{remark}

The estimate obtained above has the correct asymptotic behavior for flat weights with respect to the logarithmic $A_\infty$ constant $[w]_{A_\infty}^{\log}$, in the sense that as $[w]_{A_\infty}^{\log} \to 1$, both sides of the inequality go to zero. However, if we simply used Lemma 6.2 from \cite{HP}, which states that $[w]_{A_\infty}^{\log} \le 2^{n+1} [w]_{A_\infty}$, we would lose this property. Indeed, such a bound would yield $\norm{\log w}_{\operatorname{BMO}_w} \le c_n [w]_{A_\infty} - 1$, which clearly does not vanish as $[w]_{A_\infty} \to 1$. Therefore, we need to refine this result to obtain an inequality that preserves the correct asymptotic behavior for flat weights.

\begin{lemma}
	Let $w$ be a weight. Then, for every cube $Q$, we have
	\begin{equation*}
		\frac{1}{w(Q)} \int_Q \log ^+\left( \frac{w(x)}{w_Q}\right) w(x)\,dx \le 2^n \left( \frac{1}{w(Q)} \int_Q M(w\chi_Q)(x)\,dx - 1\right).
	\end{equation*}
	Hence, if $w\in A_\infty$, 
	\begin{equation}\label{eq:Ainftylog - Ainfty}
		[w]_{A_\infty}^{\log} -1 \le 2^n ( [w]_{A_\infty} -1).
	\end{equation}
\end{lemma}

\begin{proof}
	The estimate follows from the reverse weak-type (1,1) inequality,\begin{equation*}
		\frac{1}{t} \int_{\{ x\in Q :  w(x)>t\}} w\, dx \le 2^n \left| \{ x\in Q :  M(w\chi_Q)(x)>t\} \right|,
	\end{equation*}
	for every $t>w_Q$. By the layer-cake formula applied to $s\log^+ (\tfrac{s}{w_Q})$, we have 
	\begin{align*}
		\frac{1}{w(Q)} \int_Q \log ^+\left( \frac{w}{w_Q}\right) w = & \frac{1}{w(Q)}\int_{w_Q}^\infty \frac{1}{t}  \left( \int_{\{ x\in Q :  w(x)>t\}} w\, dx  \right) \, dt \\
		\le &  \frac{2^n}{w(Q)}\int_{w_Q}^\infty \left| \{ x\in Q :  M(w\chi_Q)(x)>t\} \right|  \,dt \\
		= &  \frac{2^n}{w(Q)}\int_{0}^\infty \left| \{ x\in Q :  M(w\chi_Q)(x)>t\} \right|  \,dt \\
		& -  \frac{2^n}{w(Q)}\int_{0}^{w_Q} \left| \{ x\in Q :  M(w\chi_Q)(x)>t\} \right|  \,dt\\
		= &  \frac{2^n}{w(Q)} \int_Q M(w\chi_Q)(x)\,dx\\
		& -  \frac{2^n}{w(Q)}\int_{0}^{w_Q} \left| \{ x\in Q :  M(w\chi_Q)(x)>t\} \right|  \,dt.
	\end{align*}
	We note that since $w_Q \le M(w\chi_Q)(x)$ for each $x\in Q$,  for every $0\le t< w_Q$ we have $\{x\in Q : M(w\chi_Q)(x) > t\} = Q$. Consequently,
	\begin{align*}
		\frac{2^n}{w(Q)}\int_{0}^{w_Q} & \left| \{ x\in Q :  M(w\chi_Q)(x)>t\} \right| \, dt \\
		= & \frac{2^n}{w(Q)} |Q| \int_0^{w_Q} dt \\
		= & 2^n \frac{|Q|}{w(Q)} w_Q \\
		= & 2^n. 
	\end{align*}
	Hence, 
	\begin{align*}
		\frac{1}{w(Q)} \int_Q \left(1+\log ^+\left( \frac{w}{w_Q}\right) \right) w \, dx - 1 = & \frac{1}{w(Q)} \int_Q \log ^+\left( \frac{w}{w_Q}\right) w\, dx\\
		 \le &  2^n \left( \frac{1}{w(Q)} \int_Q M(w\chi_Q)\,dx -1 \right),
	\end{align*}
	and we conclude the proof by taking the supremum over $Q$. 
\end{proof}

Combining the two previous results, we obtain a bound for the $\operatorname{BMO}_w$ seminorm in terms of the Fujii--Wilson $A_\infty$ constant.

\begin{corollary}\label{cor: logw in BMOw}
	If $w\in A_\infty$, then $\log w \in \operatorname{BMO}_w$ and 
	\begin{equation}\label{eq:BMOw - Ainfty}
		\norm{\log w }_{\operatorname{BMO}_w} \le 2^{n+3} ( [w]_{A_\infty} -1 ). 
	\end{equation}
	In addition, 
	\begin{equation*}
		\frac{1}{w(Q)} \int_Q |\log w(x) - (\log w)_{Q, w}| w(x)\,dx  \le 2^{n+3} \left( \frac{1}{w(Q)} \int_Q M(w\chi_Q)\,dx -1 \right),
	\end{equation*}
	for each cube $Q$. 
\end{corollary}

The following result is due to Tsutsui \cite{Tsutsui}.  However, for the convenience of the reader and the sake of completeness, we provide a detailed proof here.

\begin{theorem}\label{thm:Tsutsui}
	There exists a dimensional constant $c_n>0$ such that for every $w\in A_\infty$ and every $f\in \operatorname{BMO}_w$,
	\begin{equation}\label{eq:Tsutsui}
		\|f\|_{\operatorname{BMO}} \le c_n \log(2[w]_{A_\infty}') \|f\|_{\operatorname{BMO}_w},
	\end{equation}
	where $[w]_{A_\infty}'$ is the Hru\v{s}\v{c}ev $A_\infty$ constant.
\end{theorem}

\begin{proof}
    Fix a cube $Q$.  By standard BMO arguments, we have
    \begin{equation*}
        \frac{1}{|Q|}\int_Q |f(x)-f_Q|\,dx \le 2\inf_{c\in \R} \frac{1}{|Q|}\int_Q |f(x)-c|\,dx \le \frac{2}{|Q|}\int_Q |f(x)-f_{Q,w}|\,dx.
    \end{equation*}

    Next, by the characterization of $[w]_{A_\infty}'$ given in \cite[Proposition 7.3.2 (5)]{GrafakosCF}, for every non-negative measurable function $g$ we have
    \begin{equation*}
        \exp\left(\frac{1}{|Q|}\int_Q g(x)\,dx\right)
        \le [w]_{A_\infty}' \frac{1}{w(Q)}\int_Q \exp(g(x)) w(x)\,dx.
    \end{equation*}
    On the other hand, by the definition of the Luxemburg norm, any measurable function $h$ satisfies
    \begin{equation*}
        \frac{1}{w(Q)}\int_Q \exp\left(\frac{|h(x)|}{\|h\|_{\exp L(Q,w)}}\right) w(x)\,dx \le 2.
    \end{equation*}
    We now apply these two facts with
    \begin{equation*}
        g(x)=\frac{|f(x)-f_{Q,w}|}{\|f-f_{Q,w}\|_{\exp L(Q,w)}},
    \end{equation*}
    and $h(x)=f(x)-f_{Q,w}$. Hence, we obtain
    \begin{equation*}
        \exp\left(\frac{1}{|Q|}\int_Q \frac{ |f(x)-f_{Q,w}|} {\|f-f_{Q,w}\|_{\exp L(Q,w)}} \,dx\right) \le 2[w]_{A_\infty}'.
    \end{equation*}
    Taking logarithms gives
    \begin{equation*}
        \frac{1}{|Q|}\int_Q |f(x)-f_{Q,w}|\,dx \le \log(2[w]_{A_\infty}')\, \|f-f_{Q,w}\|_{\exp L(Q,w)}.
    \end{equation*}
	
	 Since $w\,dx$ is absolutely continuous with respect to Lebesgue measure, it vanishes on hyperplanes. Therefore, we can apply  the John--Nirenberg theorem for non-doubling measures due to Mateu, Mattila, Nicolau, and Orobitg \cite[Theorem 1]{MMNO}\footnote{As explicitly stated at the end of the proof of Theorem 1 in \cite{MMNO}, the constants in the John--Nirenberg inequality depend only on the dimension and, in particular, are independent of the measure $\mu$.}. It follows that there exists a dimensional constant $c_n>0$ such that
    \begin{equation*}
        \|f-f_{Q,w}\|_{\exp L(Q,w)} \le c_n \|f\|_{\operatorname{BMO}_w}.
    \end{equation*}
    Combining the previous estimates, we conclude that
    \begin{equation*}
        \frac{1}{|Q|}\int_Q |f(x)-f_Q|\,dx
        \le 2c_n \log(2[w]_{A_\infty}')\,\|f\|_{\operatorname{BMO}_w}.
    \end{equation*}
    Taking the supremum over all cubes $Q$ completes the proof.
\end{proof}

As a consequence of the results above, we obtain the result on the BMO norm of $\log w$ presented in Theorem \ref{thm:log w in BMO}.

\begin{proof}[Proof of Theorem \ref{thm:log w in BMO}. ]
This is just putting together \eqref{eq:BMOw - Ainfty} from Corollary \ref{cor: logw in BMOw} and Tsutsui's result in Theorem \ref{thm:Tsutsui}.
\end{proof}

\section{Proof of Theorem \ref{thm:Ainfty Ap} and its consequences}\label{sec:Ainfty Ap}

Using the results established in the previous section, we now proceed with the proof of Theorem \ref{thm:Ainfty Ap}. Our approach follows the technique of \cite{Politis} (see also \cite{Korey, Mitsis}).

\begin{proof}[Proof of Theorem \ref{thm:Ainfty Ap}]

The goal is to obtain control on the $A_p$ condition for $w\in A_\infty$ by assuming some quantitative control on $[w]_{A_\infty}$. Let's start by considering $r\ge 1$ to be chosen later and connect the $A_p$ condition 
with previous estimates on the BMO norm of $\log w$.

Let $Q$ be a cube. By Jensen's inequality and the elementary inequality $\exp(t)\le \exp (|t|)$ we obtain
	\begin{align*}
		\frac{1}{|Q|}\int_Q w \left( \frac{1}{|Q|}\int_Q w^{-r}\right)^\frac{1}{r} \le & \left( \frac{1}{|Q|}\int_Q w^r\right)^\frac{1}{r} \left( \frac{1}{|Q|}\int_Q w^{-r}\right)^\frac{1}{r}\\
		= & \left( \frac{1}{|Q|}\int_Q e^{r \log w} \right)^\frac{1}{r} \left( \frac{1}{|Q|}\int_Q e^{-r \log w}\right)^\frac{1}{r}\\
		= & \left( \frac{1}{|Q|}\int_Q e^{r (\log w - (\log w)_Q)} \right)^\frac{1}{r} \left( \frac{1}{|Q|}\int_Q e^{-r (\log w - (\log w)_Q)}\right)^\frac{1}{r}\\
		\le &\left( \frac{1}{|Q|}\int_Q e^{|r| |\log w - (\log w)_Q|} \right)^\frac{1}{r} \left( \frac{1}{|Q|}\int_Q e^{|-r| |\log w - (\log w)_Q|}\right)^\frac{1}{r}\\
		= &\left( \frac{1}{|Q|}\int_Q e^{r |\log w - (\log w)_Q|} \right)^\frac{2}{r}.
	\end{align*}
This last quantity is exactly what we know how to control by using the BMO estimates for $\log w$. More precisely, by Theorem \ref{thm:log w in BMO} we have $\log w\in \operatorname{BMO}$, and by the John--Nirenberg inequality we can obtain that
	\begin{equation*}
		|\{ x\in Q :  |\log w(x) - (\log w)_Q | > t \}| \le 2 \exp \left( - \frac{\alpha_n t }{\|\log w\|_{\operatorname{BMO}}} \right) |Q|.
	\end{equation*}
	Hence, for every $0<r\le \frac{\alpha_n}{2\|\log w\|_{\operatorname{BMO}}}$, we have
	\begin{equation}\label{eq: JN exp}
		\frac{1}{|Q|}\int_Q \exp \left( r |\log w(x) - (\log w)_Q |\right) \,dx \le 3.
	\end{equation}
	For every such $r\ge 1$, we use \eqref{eq: JN exp} to control the last integral in the previous estimate and obtain
	\begin{equation}\label{eq:Ap condition - r}
		\frac{1}{|Q|}\int_Q w \left( \frac{1}{|Q|}\int_Q w^{-r}\right)^\frac{1}{r} \le 3^\frac{2}{r}.
	\end{equation}
Two things remain to be proved: on the one hand, we need to show that $r$ above can be chosen such that $r\ge 1$. This will be a consequence of the estimates on $\|\log w\|_{\operatorname{BMO}}$. On the other hand, we need to exploit this result to obtain an explicit expression for $p([w]_{A_\infty})$, the best possible index for which $w\in A_p$.

For the first one, we simply need to restrict to weights with small constant, $[w]_{A_\infty} =1+\delta$ with $\delta$ small enough that  $\|\log w\|_{\operatorname{BMO}}\le \frac{\alpha_n}{2}$. To that end, recall that we have the estimate from Theorem \ref{thm:log w in BMO}, namely
$$
	\norm{\log w}_{\operatorname{BMO}} \le c_n \log(2[w]_{A_\infty}') ([w]_{A_\infty} - 1).
$$
Since our standing assumption is that $[w]_{A_\infty}=1+\delta$ with a small $\delta$, we may assume that $[w]_{A_\infty}$ is bounded by, let's say, 2. From that, we want to obtain that $[w]'_{A_\infty}$ is controlled as well. This can be found in \cite[Theorem 1.3]{BR14}, where the inequality 
$$
[w]'_{A_\infty}\le C \frac{e^{e^{[w]_{A_\infty}} + 1}}{e^{[w]_{A_\infty}}}
$$
is proved in dimension 1. For the higher-dimensional case in $\mathbb{R}^n$, $n\ge 2$, the result is from \cite[Corollary 5.6]{HagPar-Embeddings} where the inequality  
$$
[w]'_{A_\infty}\le e^{e^{c_n [w]_{A_\infty}}}
$$ 
has been proved. Hence, under the assumption $[w]_{A_\infty}=1+\delta$ with $\delta\le 1$, we have
$$
	\norm{\log w}_{\operatorname{BMO}} \le C_n ([w]_{A_\infty} - 1).
$$
We now choose the dimensional constant $c_n$ in the statement so that $c_n\le \min\{1,\frac{\alpha_n}{2C_n}\}$. If $\delta=0$, then $w$ is constant almost everywhere and the conclusion follows immediately. Assume that $0<\delta\le c_n$, and set $r=\frac{\alpha_n}{2C_n\delta}$. Since $\|\log w\|_{\operatorname{BMO}}\le C_n\delta$, we have $r\le \frac{\alpha_n}{2\|\log w\|_{\operatorname{BMO}}}$, so this choice is admissible in \eqref{eq: JN exp}. Moreover, $r\ge 1$. Setting $p-1=\frac{1}{r}$, we obtain
\begin{equation*}
	p=1+\frac{2C_n}{\alpha_n}([w]_{A_\infty}-1)
\end{equation*}
and \eqref{eq:Ap condition - r} yields
\begin{equation*}
	[w]_{A_p}\le 3^{2(p-1)}=3^{\frac{4C_n}{\alpha_n}([w]_{A_\infty}-1)}.
\end{equation*}
Thus, the result follows with $\tau_n=\frac{2C_n}{\alpha_n}$ and $\tau_n'=2\log 3\,\tau_n$. This concludes the proof.
\end{proof}

The only result of this type, a quantitative embedding of $A_\infty$ into some $A_p$ class, that we are aware of in the literature is the following theorem from \cite{HagPar-Embeddings}. 

\begin{theorem}
	If $w\in A_\infty$, then $w\in A_p$ with $p=e^{c_n [w]_{A_\infty}}$ and $[w]_{A_p}\le  e^{e^{c_n [w]_{A_\infty}}}$ for some dimensional constant $c_n>1$.
\end{theorem}

\begin{remark}
	We note that the corresponding result in \cite{HagPar-Embeddings} is originally formulated with the strict inequality $p > e^{c_n [w]_{A_\infty}}$. However, one can readily obtain the statement with equality, $p = e^{c_n [w]_{A_\infty}}$, by replacing $c_n$ with a slightly larger dimensional constant.
\end{remark}

Our result improves this estimate in the case $[w]_{A_\infty}\le 1+c_n$, yielding the correct asymptotic behavior for flat weights. That is, when $[w]_{A_\infty}\to 1$, we have $p\to 1$ and $[w]_{A_p}\to 1$.  Putting together our result with the embedding of \cite{HagPar-Embeddings}, we obtain the following quantitative general embedding.

\begin{corollary}
	Let $w\in A_\infty$. Then $w\in A_p$ with
	\begin{equation*}
		p = \begin{cases}
			1+\tau_n ([w]_{A_\infty}-1) & \text{ if } [w]_{A_\infty} \le 1+c_n,\\
			e^{C_n [w]_{A_\infty}} & \text{ if } [w]_{A_\infty} > 1+c_n,
		\end{cases}
	\end{equation*}
	for some dimensional constants $c_n, C_n>0$. In addition, 
	\begin{equation*}
		[w]_{A_p}\le \begin{cases}
			e^{\tau_n'  ([w]_{A_\infty} -1)} & \text{ if } [w]_{A_\infty} \le 1+c_n,\\
			e^{e^{C_n [w]_{A_\infty}}} & \text{ if } [w]_{A_\infty} > 1+c_n.
		\end{cases}
	\end{equation*}
\end{corollary}

\begin{remark}
The restriction $[w]_{A_\infty}\le 1+c_n$ is natural in our approach. It is required to guarantee that the integrability exponent $r$ obtained from the John--Nirenberg inequality is large enough (specifically, $r \ge 1$), which is a necessary condition for our estimates.    However, one may conjecture that even for large $A_\infty$ constants, the embedding could still hold with $p$ of the order $C_n [w]_{A_\infty}$.
\end{remark}

\begin{corollary}
	There exist dimensional constants $c_n, C_n>0$ such that if $w\in A_\infty$ with $[w]_{A_\infty}\le 1+c_n$, then
	\begin{equation*}
		\|f\|_{\operatorname{BMO}} \le C_n \|f\|_{\operatorname{BMO}_w}.
	\end{equation*} 
\end{corollary}

\begin{proof}
	Let $c_n$ and $\tau_n$ be the dimensional constants given by Theorem \ref{thm:Ainfty Ap}. Hence, if $[w]_{A_\infty}\le 1+c_n$, then $w\in A_p$ with $p=1+\tau_n ([w]_{A_\infty}-1)$ and $[w]_{A_p}\le C_n$. Moreover, we observe that since $[w]_{A_\infty}\le 1+c_n$, we have $D_w\le e^{e^{c_n}}$ (see Remark \ref{rem: doubling}). Let $Q$ be a cube. We have
	\begin{align*}
		\frac{1}{|Q|}\int_Q |f(x)-f_Q|dx\le & [w]_{A_p}^\frac{1}{p} \left( \frac{1}{w(Q)}\int_Q |f(x)-f_Q|^p w(x)dx\right)^\frac{1}{p}\\
		\le & [w]_{A_p}^\frac{1}{p} C_n D_w p \| f\|_{\operatorname{BMO}_w}\\
		\le & C_n \| f\|_{\operatorname{BMO}_w},
	\end{align*}
	where in the second inequality we use \cite[Theorem 1.1]{MarRivRel}, and in the last inequality we use that $p$, $[w]_{A_p}$ and $D_w$ are bounded by a dimensional constant. 
\end{proof}

\section{Applications to weighted Poincar\'e--Sobolev inequalities}\label{sec:PS}

In this section, we present some applications to weighted Poincar\'e--Sobolev inequalities, motivated by \cite{PR-Poincare}. Our first application improves \cite[Theorem 2.4]{Claros25} by incorporating the open property from Corollary \ref{cor:Ap-leftopen} into the argument developed in \cite{Claros25}, based on Riesz potentials and truncation. The resulting improvement is stated in Theorem \ref{thm:PoincareSobolev-SharpFlat1} below.

Although the results in \cite[Theorem 2.4 and Proposition 4.4]{Claros25} are stated for balls, the same proofs yield the corresponding formulations for cubes used below.
    
\begin{theorem}[\cite{Claros25}]\label{thm:Poincare-Sobolev-Claros}
Let $1\le p<n$ and $w\in A_r$ for some $1\le r\le p$. Then the following inequality holds
	\begin{equation}\label{eq:weighted PS}
		\left( \frac{1}{w(Q)}\int_Q |f-f_{Q}|^q w\right)^\frac{1}{q}\le C_w \ell(Q) \left( \frac{1}{w(Q)}\int_Q |\nabla f|^p w\right)^\frac{1}{p},	
	\end{equation}
for every cube $Q$, where the exponent $q>p$ is defined as follows. If $r>1$, let $\sigma=w^{1-r'}$ and define $q$ by
	\begin{equation*}
		\frac{1}{p}-\frac{1}{q}= \frac{1}{n}\frac{\tau_n [\sigma]_{A_\infty}}{1+r(\tau_n [\sigma]_{A_\infty}-1)},
	\end{equation*}
where $\tau_n=2^{n+1}$. If $r=1$, we set $q=p^*$.
In both cases, $C_w=c_{n} p^* [w]_{A_r}^{\frac{1}{p}}$, and the exponent $\frac{1}{p}$ in the $A_r$ constant is sharp.
\end{theorem}

We present here an improvement on this result that is particularly relevant for flat weights.

\begin{theorem}\label{thm:PoincareSobolev-SharpFlat1}
Under the same hypotheses as in Theorem \ref{thm:Poincare-Sobolev-Claros}, and assuming that $r>1$, inequality \eqref{eq:weighted PS} can be proved with exponent $q$ defined by
	\begin{equation}\label{eq:PoincareSobolev-SharpFlat1}
		\frac{1}{p}-\frac{1}{q}= \frac{1}{n}\frac{1+\tau_n ([\sigma]_{A_\infty}-1)}{1+r\tau_n ([\sigma]_{A_\infty}-1)},
	\end{equation}
and constant $C_w=c_n p^* [w]_{A_r}^{\frac{1}{p}} [\sigma]_{A_\infty}^{\frac{r-1}{p}}$.
\end{theorem}

To prove the previous result, we first recall the following result from \cite[Proposition 4.4]{Claros25}, which is motivated by \cite{ACS09} and will play a key role in the proof. Recall that, for $0<\alpha<n$, the Riesz potential is defined by
	\begin{equation*}
		I_\alpha f(x):=\int_{\R^n}\frac{f(y)}{|x-y|^{n-\alpha}}\,dy.
	\end{equation*}

\begin{proposition}[\cite{Claros25}]\label{prop:weak Riesz}
	Let $n\ge 1$, $\alpha\in (0,n)$ and $1\le p <\frac{n}{\alpha}$. Let $w\in A_p$ and consider $q_r$ defined by the relation 
	\begin{equation*}
		\frac{1}{p}-\frac{1}{q_r}=\frac{\alpha}{n} \frac{1}{r},
	\end{equation*}
	with $1\le r\le p$. If $w\in A_r$, then there exists a dimensional constant $c_n>0$ such that 
	\begin{equation*}
		\left\| I_\alpha f \right\|_{L^{q_r, \infty}\left(Q, \frac{w(x)dx}{w(Q)}\right)} \le \frac{c_n}{\alpha} p^*_\alpha [w]_{A_r}^\frac{1}{p} \ell(Q)^\alpha \left( \frac{1}{w(Q)}\int_Q |f(x)|^p w(x)\, dx\right)^\frac{1}{p},
	\end{equation*}
	for any cube $Q$ and every $f\in L^p(Q, w)$, where $p^*_\alpha = \frac{np}{n-\alpha p}$. 
\end{proposition}

\begin{proof}[Proof of Theorem \ref{thm:PoincareSobolev-SharpFlat1}]
	By Corollary \ref{cor:Ap-leftopen}, we have $w\in A_s$ with $[w]_{A_s}\le (2[\sigma]_{A_\infty})^{r-1}[w]_{A_r}$, where $s<r$ is given by
	\begin{equation*}
		s=r-\frac{r-1}{1+\tau_n([\sigma]_{A_\infty}-1)}=\frac{1+r\tau_n([\sigma]_{A_\infty}-1)}{1+\tau_n([\sigma]_{A_\infty}-1)}.
	\end{equation*}
	Moreover,
	\begin{equation*}
		\frac{1}{p}-\frac{1}{q}=\frac{1}{ns}=\frac{1}{n}\frac{1+\tau_n([\sigma]_{A_\infty}-1)}{1+r\tau_n([\sigma]_{A_\infty}-1)},
	\end{equation*}
	which gives the exponent in \eqref{eq:PoincareSobolev-SharpFlat1}. Fix a cube $Q$. Since $1\le s<r\le p$, using the subrepresentation formula $|f(x)-f_Q|\le c_n I_1(|\nabla f|\chi_Q)(x)$ for almost every $x\in Q$  (see for instance \cite[Theorem 15.16]{WZ}) and Proposition \ref{prop:weak Riesz} with $\alpha=1$, we obtain
	\begin{align*}
		\left\| f-f_Q\right\|_{L^{q, \infty}\left(Q, \frac{w(x)dx}{w(Q)}\right)}\le & c_n \left\| I_1(|\nabla f|\chi_Q) \right\|_{L^{q, \infty}\left(Q, \frac{w(x)dx}{w(Q)}\right)}\\
		\le &  c_n p^* [w]_{A_s}^\frac{1}{p} \ell(Q) \left( \frac{1}{w(Q)}\int_Q |\nabla f(x)|^p w(x)\, dx\right)^\frac{1}{p}\\
		\le &  c_n p^* [w]_{A_r}^{\frac{1}{p}} [\sigma]_{A_\infty}^{\frac{r-1}{p}} \ell(Q) \left( \frac{1}{w(Q)}\int_Q |\nabla f(x)|^p w(x)\, dx\right)^\frac{1}{p}.
	\end{align*}
	
	Finally, we conclude the proof by applying the well-known truncation argument (see \cite{Ha, KO03}). As is implicit in \cite{FPW98}, this method can be applied to the local inequality while maintaining the unweighted average $f_Q$ on the left-hand side; see also \cite[Proposition A.1]{LoristWagenaar} for an explicit proof. More precisely, the weak-type estimate above can be upgraded to the desired strong-type inequality provided that the corresponding unweighted $L^1$ Poincar\'e inequality holds with the same right-hand side. Thus, it only remains to verify this auxiliary estimate. The classical $(1,1)$ Poincar\'e inequality and the $A_p$ condition yield
	\begin{align*}
		\frac{1}{|Q|}\int_Q |f(x)-f_Q|\, dx \le & c_n \ell(Q) \frac{1}{|Q|}\int_Q |\nabla f (x)|\, dx \\
		\le & c_n [w]_{A_p}^\frac{1}{p} \ell(Q) \left( \frac{1}{w(Q)}\int_Q |\nabla f(x)|^p w(x) \, dx \right)^\frac{1}{p}\\
		\le & c_n [w]_{A_r}^{\frac{1}{p}} [\sigma]_{A_\infty}^{\frac{r-1}{p}} \ell(Q) \left( \frac{1}{w(Q)}\int_Q |\nabla f(x)|^p w(x)\, dx\right)^\frac{1}{p},
	\end{align*}
	where we have used that $s\le p$ and hence $[w]_{A_p}\le [w]_{A_s}\le (2[\sigma]_{A_\infty})^{r-1}[w]_{A_r}$. 	This proves the required auxiliary estimate, and hence the result follows from \cite[Proposition A.1]{LoristWagenaar}.
\end{proof}
 
We remark here that, although the constant becomes larger and is no longer sharp, the resulting Sobolev exponent is correct for flat weights; that is, it exhibits the correct asymptotic behavior as $[\sigma]_{A_\infty}\to 1$. Note that the result above is valid for every $A_r$ weight with $r>1$, without any smallness assumption on $[w]_{A_\infty}$ or $[\sigma]_{A_\infty}$, although it is particularly relevant for flat weights. However, the result also depends on the $A_r$ class: the choice of $r$ for which $w\in A_r$ enters quantitatively both in the exponent $q$ in \eqref{eq:PoincareSobolev-SharpFlat1} and in the constant in the inequality.

Now we provide the proof of the main result on Poincar\'e--Sobolev inequalities presented in Theorem \ref{thm:PoincareSobolev-SharpFlat2}. We will apply our embedding result (Theorem \ref{thm:Ainfty Ap}) to prove weighted Poincar\'e--Sobolev inequalities for $A_\infty$ weights with small $A_\infty$ constant.

\begin{proof}[Proof of Theorem \ref{thm:PoincareSobolev-SharpFlat2}]
	Let $c_n$ be the dimensional constant from Theorem \ref{thm:Ainfty Ap}, and set $c_{n,p} := \min\{ c_n, (p-1)\tau_n^{-1}\}$. Let $w\in A_\infty$ such that $[w]_{A_\infty}\le 1+c_{n,p}$. Since $c_{n,p}\le c_n$, Theorem \ref{thm:Ainfty Ap} gives $w\in A_r$ with $r=1+\tau_n([w]_{A_\infty}-1)$ and $[w]_{A_r}\lesssim_n 1+\tau_n([w]_{A_\infty}-1)$. Moreover, we have $r \le 1+\tau_n c_{n,p} \le p$. Thus $w\in A_r\subset A_p$ with $1\le r\le p$, and we can apply Proposition \ref{prop:weak Riesz} with $\alpha=1$, obtaining $q_r=p^*_w$. For each cube $Q$, using the subrepresentation formula $|f(x)-f_Q|\le C_n I_1(|\nabla f |\chi_Q)(x)$ for almost every $x\in Q$, we have
	\begin{align*}
		\left\| f-f_Q\right\|_{L^{p^*_w, \infty}\left(Q, \frac{w(x)dx}{w(Q)}\right)}\le & c_n \left\| I_1 (|\nabla f|\chi_Q)  \right\|_{L^{q_r, \infty}\left(Q, \frac{w(x)dx}{w(Q)}\right)}\\
		\le & C_n p^* [w]_{A_r}^\frac{1}{p} \ell(Q) \left( \frac{1}{w(Q)}\int_Q |\nabla f|^p w\right)^\frac{1}{p}	\\
		\le & C_n p^* (1+\tau_n ([w]_{A_\infty}-1))^\frac{1}{p} \ell(Q) \left( \frac{1}{w(Q)}\int_Q |\nabla f|^p w\right)^\frac{1}{p}.	
	\end{align*}
	Finally, we conclude the proof by applying the truncation method from \cite[Theorem 2.3]{CPZ}, which is a Lorentz-space version of the weak-to-strong principle in \cite[Proposition A.1]{LoristWagenaar} and follows the ideas in \cite{FPW98}. The additional unweighted $L^1$ estimate required in this result is obtained exactly as in the proof of Theorem \ref{thm:PoincareSobolev-SharpFlat1}, using the classical $(1,1)$ Poincar\'e inequality and the $A_p$ condition. We thus obtain the claimed $L^{p^*_w,p}$ estimate, and the strong $L^{p^*_w}$ estimate follows from the inclusion $L^{p^*_w,p}\hookrightarrow L^{p^*_w}$, since $p<p^*_w$.
\end{proof}

\bibliographystyle{amsalpha}
\providecommand{\bysame}{\leavevmode\hbox to3em{\hrulefill}\thinspace}
\providecommand{\MR}{\relax\ifhmode\unskip\space\fi MR }
\providecommand{\MRhref}[2]{%
  \href{http://www.ams.org/mathscinet-getitem?mr=#1}{#2}
}
\providecommand{\href}[2]{#2}


\end{document}